\documentclass[a4paper,11pt]{article}

\usepackage{amsmath,amsthm,amssymb,mathrsfs,latexsym,amsfonts}
\usepackage{graphicx,psfrag,epsfig}
\usepackage{bbm}
\usepackage[english]{babel}
\usepackage[latin1]{inputenc}
\usepackage{a4wide}

\newtheorem{theorem}{Theorem}

\newtheorem{proposition}{Proposition}

\newcounter{paraga}[section]

\begin{document}

\def\MP{\,{<\hspace{-.5em}\cdot}\,}
\def\SP{\,{>\hspace{-.3em}\cdot}\,}
\def\PM{\,{\cdot\hspace{-.3em}<}\,}
\def\PS{\,{\cdot\hspace{-.3em}>}\,}
\def\EP{\,{=\hspace{-.2em}\cdot}\,}
\def\PP{\,{+\hspace{-.1em}\cdot}\,}
\def\PE{\,{\cdot\hspace{-.2em}=}\,}
\def\N{\mathbb N}
\def\C{\mathbb C}
\def\Q{\mathbb Q}
\def\R{\mathbb R}
\def\T{\mathbb T}
\def\A{\mathbb A}
\def\Z{\mathbb Z}
\def\demi{\frac{1}{2}}

\begin{titlepage}
\author{Abed Bounemoura~\footnote{CNRS-IMPA UMI 2924  (abedbou@gmail.com)}}
\title{\LARGE{\textbf{Ergodization time for linear flows on tori via geometry of numbers}}}
\end{titlepage}

\maketitle

\begin{abstract}
In this paper, we give a new, short, simple and geometric proof of the optimal ergodization time for linear flows on tori. This result was first obtained by Bourgain, Golse and Wennberg in \cite{BGW98} using Fourier analysis. Our proof uses geometry of numbers: it follows trivially from a Diophantine duality that was established by the author and Fischler in \cite{BF13}.  
\end{abstract}

\bigskip
\bigskip

Let $n \geq 2$ be an integer, $\T^n:=\R^n / \Z^n$, $\alpha \in \R^n\setminus\{0\}$ and consider the linear flow on $\T^n$ defined by
\[ X_\alpha^t(\theta)=\theta+t\alpha, \quad t \in \R, \quad \theta \in \T^n.  \]
It is just the flow of the constant vector field $X_\alpha=\alpha$ on $\T^n$. Such flows play an important role in Hamiltonian systems, and their dynamical properties depend on the Diophantine properties of the vector $\alpha$, as we will recall now.

Let us say that a vector subspace of $\R^n$ is rational if it admits a basis of vectors with rational components. We define $F_\alpha$ to be the smallest rational subspace of $\R^n$ containing $\alpha$, so that $\Lambda_\alpha:=F_\alpha \cap \Z^n$ is a lattice in $F_\alpha$. In the special case where $F_\alpha=\R^n$, we have $\Lambda_\alpha=\Z^n$ and the vector $\alpha$ is said to be non-resonant: it is an elementary fact that the flow $X_\alpha^t$ is minimal (all orbits are dense) and in fact uniquely ergodic (all orbits are uniformly distributed with respect to Haar measure). In the general case where $F_\alpha$ has dimension $d$ with $1 \leq d \leq n$, choosing a complementary subspace $E_\alpha$ of $F_\alpha$, the affine foliation
\[ \R^n=\bigsqcup_{v \in E_\alpha}v+F_\alpha \]
induces a foliation on $\T^n$ such that each leaf, which is just a translate of the $d$-dimensional torus $\mathcal{T}^d_\alpha:=F_\alpha/\Lambda_\alpha$, is invariant by the flow and the restriction of the latter is minimal and uniquely ergodic.

A natural question is the following. Given some $T>0$, let 
\begin{equation}\label{piece}
\mathcal{O}_\alpha^T:=\bigcup_{0 \leq t \leq T}X_\alpha^t(0) \subset \mathcal{T}_\alpha^d 
\end{equation}
be a finite piece of orbit starting at the origin. As $T$ goes to infinity, $\mathcal{O}_\alpha^T$ fills the torus $\mathcal{T}_\alpha^d$, hence given any $\delta>0$, there exists a smallest positive time $T_\alpha(\delta)$ such that $\mathcal{O}_\alpha^{T_\alpha(\delta)}$ is a $\delta$-dense subset of $\mathcal{T}_\alpha^d$ (for a fixed metric on $\mathcal{T}_\alpha^d$ induced by a choice of a norm on $F_\alpha$). This time $T_\alpha(\delta)$ is usually called the $\delta$-ergodization time. Note that in~\eqref{piece} we chose the initial condition $\theta_0=0$; yet it is obvious that choosing a different $\theta_0$ (and consequently a different leaf of the foliation) lead to the same ergodization time. Then the question is to estimate this time $T(\delta)$ as a function of $\delta$.

To do so, let us define the function
\begin{equation}\label{funcpsi}
\Psi_{\alpha}(Q):=\max\{|k \cdot \alpha|^{-1} \; | \; k \in \Lambda_\alpha, \: 0<|k|\leq Q \},
\end{equation}
where, if $k=(k_1,\dots,k_n)$ and $\alpha=(\alpha_1,\dots,\alpha_n)$,
\[ k\cdot\alpha=k_1\alpha_1+\cdots+k_n\alpha_n, \quad |k|=\max\{|k_1|,\dots,|k_n|\}. \]
The function $\Psi$ in~\eqref{funcpsi} is defined for $Q \geq Q_\alpha$, where $Q_\alpha \geq 1$ is the length of the shortest non-zero vector in $\Lambda_\alpha$, that is
\[ Q_{\alpha}:=\inf\{|k| \; | \; k\in \Lambda_\alpha\setminus\{0\}\} \]
which depends only on the lattice. Another lattice constant is the co-volume $|\Lambda_\alpha|$ of $\Lambda_\alpha$, which is the $d$-dimensional volume of a fundamental domain of $\mathcal{T}_\alpha^d=F_\alpha/\Lambda_\alpha$, and let us write
\[ C_\alpha:=|\Lambda_\alpha|^2. \]     
Without loss of generality, we may assume that the vector $\alpha$ has a component equals to one; if not, just divide $\alpha$ by $|\alpha|$, and changing its sign if necessary, one just needs to consider $T_\alpha(\delta)/|\alpha|$ instead of $T_\alpha(\delta)$.

We can now state our main result.

\begin{theorem}\label{thm1}
Let $\delta>0$ such that $\delta \leq d^2((n+2)Q_\alpha)^{-1}$. Then we have the inequality
\[ T_\alpha(\delta) \leq C_{d,\alpha}\Psi\left(2C_{d,\alpha}\delta^{-1}\right), \quad C_{d,\alpha}:=d^2d!C_\alpha. \]
\end{theorem}

Even though, up to our knowledge, this result haven't been stated and proved in this generality, it is not essentially new; the novelty lies in its proof.

But first let us recall the previous results, which were dealing only with the case $d=n$ (in this case, one has $Q_\alpha=C_\alpha=1$ and there is no restriction on $0<\delta \leq 1$). Assuming moreover that $\alpha$ satisfies the following Diophantine condition:
\[ |k\cdot\alpha| \geq \gamma|k|^{-\tau}, \quad k \in \Z^n\setminus\{0\}, \quad \gamma>0, \quad \tau\geq n-1, \]
the above result reads
\[ T_\alpha(\delta) \lesssim \delta^{-\tau}. \]
This result, but with the exponent $\tau$ replaced by the worse exponent $\tau+n$, was first established in~\cite{CG94}, where it was used in the problem of instability of Hamiltonian systems close to integrable (Arnold diffusion). This was then slightly improved in \cite{Dum91} to the value $\tau+n/2$; see also \cite{DDG96} where this ergodization time is shown to be closely related to problems in statistical physics. The estimate with the exponent $\tau$ was eventually obtained in~\cite{BGW98}, and then later the more general statement (without assuming $\alpha$ to be Diophantine) was obtained in~\cite{BBB03}. For a survey on the results and applications of this ergodization time, we refer to~\cite{Dum99}. 

All these proofs are based on Fourier analysis, and it is our purpose to offer a new proof, which is geometric and rather simple. First let us observe that there is a trivial case, namely when $d=1$. In this case, writing $\alpha=\omega$, the vector is in fact rational, that is $q\omega \in \Z^n$ for some minimal integer $q\geq 1$, and obviously $T_\omega(\delta)=q$ for any $0 \leq \delta \leq 1$: in fact, the orbits of the linear flow $X_\omega^t$ are all periodic of period $q$ so for $T=q$, one has the equality $\mathcal{O}_\omega^T=\mathcal{T}^1_\omega$. Our proof will essentially reduce the general case to the trivial case: the proposition below shows that in general the linear flow $X_\alpha^t$ can be approximated by $d$ periodic flows $X_{\omega_j}^t$ with periods $q_j$, $1 \leq j \leq d$, and such that the vectors $q_j\omega_j \in \Z^n$ span the lattice $\Lambda_\alpha$. Here's a precise statement.

\begin{proposition}\label{prop}
Let $Q \geq (n+2)Q_\alpha$. Then there exist $d$ rational vectors $\omega_1, \dots, \omega_d$ in $\Q^n$, of denominators $q_1, \dots, q_d$, such that:
\begin{itemize}
\item[$(i)$] for all $1 \leq j \leq d$, $|\alpha-\omega_j|\leq d(q_jQ)^{-1}$ and $1 \leq q_j \leq dd!C_\alpha \Psi(2d!C_\alpha Q)$;
\item[$(ii)$] the integer vectors $q_1\omega_1, \dots, q_d\omega_d$ form a basis for the lattice $\Lambda_\alpha$.
\end{itemize}
\end{proposition}

This Proposition was proved in \cite{BF13}, see Theorem 2.1 and Proposition 2.3. The only ingredient used there is the following well-known transference result in geometry of numbers (see \cite{Cas} for instance): if $\mathcal{C}$ and $\Lambda$ are respectively a convex body and a lattice in a Euclidean space of dimension $d$, and if $\mathcal{C}^*$ and $\Lambda^*$ denote their dual, then
\[ 1 \leq \lambda_k(\mathcal{C},\Lambda)\lambda_{d+1-k}(\mathcal{C}^*,\Lambda^*)\leq d!, \quad 1 \leq k \leq d \] 
where $\lambda_k(\mathcal{C},\Lambda)$ is the $k$-th successive minima of $\mathcal{C}$ with respect to $\Lambda$.  

The proof of Theorem~\ref{thm1} is now a trivial matter if one uses Proposition~\ref{prop}.

\begin{proof}[Proof of Theorem~\ref{thm1}]
Choose $Q=d^2\delta^{-1}$. Since we required $\delta \leq d^2((n+2)Q_\alpha)^{-1}$, $Q \geq (n+2)Q_\alpha$ and so Proposition~\ref{prop} can be applied. It follows from $(ii)$ that the set
\[ \{t_1q_1\omega_1+\cdots+t_dq_d\omega_d \; | \; (t_1,\dots,t_d) \in [0,1[^d\} \]
is a fundamental domain for $\mathcal{T}_\alpha^d=F_\alpha/\Lambda_\alpha$. Hence given an arbitrary point $\theta^* \in \mathcal{T}_\alpha^d$, there is a unique $(t_1^*,\dots,t_d^*) \in [0,1[^d$ such that 
\[ \theta^*=t_1^*q_1\omega_1+\cdots+t_d^*q_d\omega_d.  \]
Now by $(i)$, for any $1 \leq j \leq d$, we have
\begin{equation}\label{est}
|t_j^*q_j\alpha-t_j^*q_j\omega_j|\leq dt_j^*Q^{-1} \leq dQ^{-1}, \quad 1 \leq q_j \leq dd!C_\alpha \Psi(2d!C_\alpha Q),
\end{equation}
so that if we set $T^*=t_1^*q_1+\cdots+t_d^*q_d$, the first inequality of~\eqref{est} gives
\[ |T^*\alpha-\theta^*|=\left|\left(\sum_{j=1}^{d}t_j^*q_j\right)\alpha-\sum_{j=1}^{d}t_j^*q_j\omega_j\right|\leq \sum_{j=1}^{d}|t_j^*q_j\alpha-t_j^*q_j\omega_j|\leq d^2Q^{-1}=\delta \]
while the second inequality of~\eqref{est} gives
\[ T^*=\sum_{j=1}^{d}t_j^*q_j \leq \sum_{j=1}^{d}q_j \leq d^2d!C_\alpha \Psi(2d!C_\alpha Q)=d^2d!C_\alpha \Psi(2d^2d!C_\alpha \delta^{-1})=C_{d,\alpha}\Psi(2C_{d,\alpha}\delta^{-1}).  \]
The result follows.   
\end{proof}

To conclude, let us examine the special case $n=2$, that is we consider the linear flow associated to $(1,\alpha) \in \R^2$, with $|\alpha| \leq 1$. By the classical processes of suspension and taking section, it is equivalent to consider the circle rotation $R_\alpha : \T \rightarrow \T$ given by $R_\alpha(x)=x+\alpha$ mod $1$. For any $0<\delta<1$, the $\delta$-ergodization time of $R_\alpha$ is the smallest natural number $N=N_\alpha(\delta)$ such that for any $x\in \T$, the finite orbit $\{x,R_{\alpha}(x), \dots, R_\alpha^N(x)\}$ is a $\delta$-dense subset of $\T$. Observe that for $\alpha\notin \Q$, $N_\alpha(\delta)$ is always well defined, while for $\alpha=p/q\in \Q^*$, $N_{p/q}(\delta)$ is well-defined if and only if $\delta\geq q^{-1}$ in which case $N_{p/q}(\delta)\leq q-1$.

Classical proofs in the special case $n=2$ are usually based on continued fractions (see \cite{DDG96} for instance), and there was a belief that the absence of a good analog of continued fractions in many dimension was an obstacle to extend the known estimate for $n=2$. We take the opportunity here to give an elementary proof in the case $n=2$ which does not use continued fractions but simply relies on the Dirichlet's box principle. For simplicity, we shall write $\Psi_\alpha=\Psi_{(1,\alpha)}$ the function defined in~\eqref{funcpsi}.

\begin{theorem}
If $\alpha \notin \Q$ and $|\alpha|\leq 1$, we have 
\[ N_\alpha(\delta)\leq [\Psi_\alpha(2\delta^{-1})]-1\] where $[\,.\,]$ denotes the integer part.
\end{theorem}

\begin{proof}
Recall that by Dirichlet's box principle, given any $Q\geq 1$, there exists $(q,p)\in \N^* \times \Z$ relatively prime such that
\[ |q\alpha-p|\leq Q^{-1}, \quad 1 \leq q \leq Q. \]
Apply this with $Q=\Psi_\alpha(2\delta^{-1})$: there exists $(q,p)\in \N^* \times \Z$ relatively prime such that 
\[ |q\alpha-p|\leq \Psi_\alpha(2\delta^{-1})^{-1}, \quad 1 \leq q \leq \Psi_\alpha(2\delta^{-1}). \]
From the second estimate, $q \leq [\Psi_\alpha(2\delta^{-1})]$, so it is enough to show that $N_\alpha(\delta)\leq q-1$. From the first estimate above, the definition of $\Psi_\alpha$ and the fact that $\max\{|q|,|p|\}=q$ (as $|\alpha|\leq 1$), we have $\Psi_\alpha(q) \geq \Psi_\alpha(2\delta^{-1})$, hence $q\geq 2\delta^{-1}$, $\delta/2 \geq q^{-1}$ and so $N_{p/q}(\delta/2)\leq q-1$. Using the first estimate again and the fact that $\Psi_\alpha(2\delta^{-1})^{-1}\leq \delta/2$, it is easy to see that the distance between $\{x,R_{\alpha}(x), \dots, R_\alpha^{q-1}(x)\}$ and $\{x,R_{p/q}(x), \dots, R_{p/q}^{q-1}(x)\}$ is at most $\delta/2$. The latter set being $\delta/2$-dense, the former is $\delta$-dense, hence $N_\alpha(\delta) \leq q-1$ and this ends the proof.  
\end{proof}

\addcontentsline{toc}{section}{References}
\bibliographystyle{amsalpha}
\bibliography{Ergod}

\providecommand{\bysame}{\leavevmode\hbox to3em{\hrulefill}\thinspace}
\providecommand{\MR}{\relax\ifhmode\unskip\space\fi MR }
\providecommand{\MRhref}[2]{%
  \href{http://www.ams.org/mathscinet-getitem?mr=#1}{#2}
}
\providecommand{\href}[2]{#2}
\begin{thebibliography}{BGW98}

\bibitem[BBB03]{BBB03}
M.~Berti, L.~Biasco, and P.~Bolle, \emph{Drift in phase space: a new
  variational mechanism with optimal diffusion time}, J. Math. Pures Appl. (9)
  \textbf{82} (2003), no.~6, 613--664.

\bibitem[BF13]{BF13}
A.~Bounemoura and S.~Fischler, \emph{A diophantine duality applied to the {KAM}
  and {N}ekhoroshev theorems}, Math. Z. \textbf{275} (2013), no.~3, 1135--1167.

\bibitem[BGW98]{BGW98}
J.~Bourgain, F.~Golse, and B.~Wennberg, \emph{On the distribution of free path
  lengths for the periodic {L}orentz gas}, Comm. Math. Phys. \textbf{190}
  (1998), no.~3, 491--508.

\bibitem[Cas59]{Cas}
J.W.S. Cassels, \emph{An introduction to the geometry of numbers}, Grundlehren
  der Math. Wiss., no.~99, Springer, 1959.

\bibitem[CG94]{CG94}
L.~Chierchia and G.~Gallavotti, \emph{Drift and diffusion in phase space}, Ann.
  Inst. Henri Poincaré, Phys. Théor. \textbf{60} (1994), no.~1, 1--144.

\bibitem[DDG96]{DDG96}
H.~S. Dumas, L.~Dumas, and F.~Golse, \emph{On the mean free path for a periodic
  array of spherical obstacles}, J. Statist. Phys. \textbf{82} (1996), no.~5-6,
  1385--1407.

\bibitem[Dum91]{Dum91}
H.~S. Dumas, \emph{Ergodization rates for linear flow on the torus}, J. Dynam.
  Differential Equations \textbf{3} (1991), no.~4, 593--610.

\bibitem[Dum99]{Dum99}
\bysame, \emph{Filling rates for linear flow on the torus: recent progress and
  applications}, Hamiltonian systems with three or more degrees of freedom
  ({S}'{A}gar\'o, 1995), NATO Adv. Sci. Inst. Ser. C Math. Phys. Sci., vol.
  533, Kluwer Acad. Publ., Dordrecht, 1999, pp.~335--339.

\end{thebibliography}

\end{document}